\documentclass[10pt]{article}%
\usepackage{amsmath}
\usepackage{amsfonts}
\usepackage{mathrsfs}
\usepackage{amssymb, color}
\usepackage[linkcolor=black,anchorcolor=black,citecolor=black]{hyperref}
\usepackage{graphicx}
\numberwithin{equation}{section}
\usepackage[body={15.5cm,21cm}, top=3cm]{geometry}%
\providecommand{\U}[1]{\protect\rule{.1in}{.1in}}
\providecommand{\U}[1]{\protect \rule{.1in}{.1in}}
\newtheorem{theorem}{Theorem}[section]

\newtheorem{definition}[theorem]{Definition}

\newtheorem{lemma}[theorem]{Lemma}

\newtheorem{proposition}[theorem]{Proposition}
\newtheorem{remark}[theorem]{Remark}

\newenvironment{proof}[1][Proof]{\noindent \textbf{#1.} }{\  \rule{0.5em}{0.5em}}

\begin{document}
\title{Reflected BSDE driven by $G$-Brownian motion with an upper obstacle}
\author{ Hanwu Li\thanks{School of Mathematics, Shandong University,
lihanwu@mail.sdu.edu.cn.}
\and Shige Peng\thanks{School of Mathematics and Qilu Institute of Finance, Shandong University,
peng@sdu.edu.cn. Li and Peng's research was
partially supported by the Tian Yuan Projection of the National Nature Sciences Foundation of China (No. 11526205 and No. 11626247)  and by the 111
Project (No. B12023).}}
\date{}
\maketitle
\begin{abstract}
In this paper, we study the reflected backward stochastic differential equation driven by $G$-Brownian motion (reflected $G$-BSDE for short) with an upper obstacle. The existence is proved by approximation via penalization. By using a variant comparison theorem, we show that the solution we constructed is the largest one.
\end{abstract}


\textbf{MSC-classification}: 60H10, 60H30
\section{Introduction}
Linear backward stochastic differential equations (BSDEs for short) were initiated by Bismut \cite{B}. Then Pardoux and Peng \cite{PP} studied the general nonlinear case. Roughly speaking, on a filtered probability space $(\Omega,\mathcal{F},\{\mathcal{F}_t\}_{0\leq t\leq T},\mathbb{P})$ generated by a Brownian motion $B$, a solution to a BSDE is a couple $(Y,Z)$ of progressively measurable processes satisfying:
\[Y_t=\xi+\int_t^T f(s,Y_s,Z_s)ds-\int_t^T Z_sdB_s, \ 0\leq t\leq T,\ \mathbb{P}\textrm{-}a.s.,\]
where the generator is progressively measurable and the terminal value $\xi$ is an $\mathcal{F}_T$-measurable random variable. Pardoux and Peng obtained the existence and uniqueness of the above equation when $f$ is uniformly Lipschitz and both $f$ and $\xi$ are square integrable. Because it can be widely applied in mathematical finance, stochastic control, stochastic differential games and probabilistic method for partial differential equations, the BSDE theory has attracted considerable attention.

Reflected backward stochastic differential equations (RBSDE for short) were firstly studied by El Karoui, Kapoudjian, Pardoux, Peng and Quenez \cite{KKPPQ}. The solution $Y$ of RBSDE is required to be above a given continuous process $S$ so that an additional increasing process should be added in the equation. This increasing process should satisfy the Skorohod condition, which ensures that it behaves in a minimal way, i.e., it only acts when $Y$ reaches the obstacle $S$. This theory provides a useful method for pricing American contingent claims, see \cite{EPQ}. It also gives a probabilistic representation for the solution of an obstacle problem for nonlinear parabolic PDE, which establishes the connection with variational equalities, see \cite{BCFE} and \cite{KKPPQ}.

Motivated by probabilistic interpretation for fully nonlinear PDEs and applications in financial markets in the uncertainty volatility model (UVM for short), Peng \cite{P04,P05} systemically established a time-consistent fully nonlinear expectation theory. As a typical case, Peng introduced the $G$-expectation (see \cite{P10} and the reference therein). Under $G$-expectation framework, a new type of Brownian motion $\{B_t\}_{t\geq 0}$, called $G$-Brownian motion, was constructed. Different from the classical case, its quadratic variation process $\langle B\rangle$ is not deterministic. The stochastic integrals with respect to $B$ and $\langle B\rangle$ were also established. Similar with the classical SDE theory, the existence and uniqueness of solution of a stochastic differential equation driven by $G$-Brownian motion ($G$-SDE) can be proved by the contracting mapping theorem. The challenging and fascinating problem of wellposedness for BSDE driven by $G$-Brownian motion has been solved by Hu et al. \cite{HJPS1}. In their paper, they showed that there exists a unique triple $(Y,Z,K)$ in proper Banach spaces satisfying the following equation:
\begin{displaymath}
Y_t=\xi+\int_t^Tf(s,Y_s,Z_s)ds+\int_t^Tg(s,Y_s,Z_s)d\langle B\rangle_s-\int_t^T Z_s dB_s-(K_T-K_t).
\end{displaymath}
In the accompanying paper \cite{HJPS2}, the comparison theorem, nonlinear Feymann-Kac formula and Girsanov transformation were established. We should point out that the equation above holds $P$-$a.s.$ for every probability measure $P$ that belongs to a non-dominated class of mutually singular measures. Therefore, the $G$-BSDE is highly related to the second order BSDEs (2BSDEs for short) developed by Cheridito, Soner, Touzi and Victoir \cite{CSTV}. The advantage of the $G$-BSDE theory is that the solution $(Y,Z,K)$ is universally defined in the spaces of the $G$-framework and the processes have strong regularity property.

In the past two decades, a great deal of effort have been devoted to the study of various types of RBSDEs. Cvitanic and Karaztas \cite{CK} and Hamadene and Lepeltier \cite{HL} generalized the results above to the case of two reflecting obstacles and established the connection between this problem and that of Dynkin games. Hamadene \cite{H} and Lepeltier and Xu \cite{LX} gave a generalized Skorohod condition and obtained a wellposedness theory when the obstacle process has c\`{a}dl\`{a}g paths.  Matoussi, Possamai and Zhou \cite{MPZ} showed the existence and uniqueness of second order reflected BSDEs with a lower obstalce.

Recently, Li and Peng \cite{lp} introduced the notion of reflected $G$-BSDE with a lower obstacle. In order to make sure that the solution $Y$ can be pushed upwardly so that it is above the given continuous process $S$, called lower obstacle, an increasing process will be added in this equation. Due to the appearance of the decreasing $G$-martingale in $G$-BSDE, we should change the Skorohod condition slightly to the ``martingale condition". The uniqueness can be derived from a priori estimates and we use the approximation via penalization to solve the existence. More precisely, consider the following $G$-BSDEs parameterized by $n=1,2,\cdots$,
\[
Y_t^n=\xi+\int_t^T f(s,Y_s^n,Z_s^n)ds+\int_t^T g(s,Y_s^n,Z_s^n)d\langle B\rangle_s
-\int_t^T Z_s^ndB_s-(K_T^n-K_t^n)+(L_T^n-L_t^n), \tag{a}
\]
where $L_t^n=n\int_0^t(Y_s^n-S_s)^-ds$. We claim that the solution $(Y,Z,A)$ of the reflected $G$-BSDE with parameters $(\xi,f,g,S)$ is the limit of $(Y^n,Z^n,L^n-K^n)$. The proof of convergence in appropriate spaces becomes delicate and challenging. The difficulty lies in the fact that the Fatou's lemma cannot be directly and automatically used in this sublinear expectation framework. Besides,  any bounded sequence in  $M_G^\beta(0,T)$  is not weakly compact. It is worth mentioning that the uniformly continuous property of the elements in $S_G^p(0,T)$ plays a key role in overcoming this problem.

In the classical situation, the solution $(Y,Z,L)$ of reflected BSDE with terminal value $\xi$, generator $f$ and upper obstacle $S$ corresponds to $(-\tilde{Y},-\tilde{Z},-\tilde{L})$. Here $(\tilde{Y},\tilde{Z},\tilde{L})$ is the solution of reflected BSDE with data $(-\xi,\tilde{f},-S)$, where $\tilde{f}(s,y,z)=-f(s,-y,-z)$. To obtain the existence result for reflected $G$-BSDE with an upper obstacle, applying the penalization method, we need to replace the increasing process $\{L_t^n\}$ in the penalized $G$-BSDE (a) by a decreasing one $\{\widetilde{L}_t^n\}=\{-n\int_0^t(Y_s^n-S_s)^+ds\}$ such that the solution can be pulled downward to be below the given continuous obstacle process. Since there will be a decreasing $G$-martingale $\{K_t^n\}$, these two cases are significantly different under the $G$-framework: $\{L_t^n-K_t^n\}$ is an increasing process while $\{\widetilde{L}_t^n-K_t^n\}$ is a finite variation process. Therefore, we should modify the conditions on the generators and the obstacle process. In the lower obstacle case, we prove the uniform bounded property of sequences $\{Y^n\}$, $\{L^n\}$, $\{K^n\}$ simultaneously by using $G$-It\^{o}'s formula and then get the uniform convergence of $\{(Y^n-S)^-\}$. However, for the upper obstacle case, we will show the rate of convergence of $\{(Y^n-S)^+\}$ in order to derive the uniform bounded property for $\{L^n\}$ and $\{K^n\}$ respectively. Furthermore, the solution to this problem by our construction is proved to be the largest one using a variant comparison theorem.

The rest of paper is organized as follows. Section 2 is devoted to listing some notations and results as preliminaries for the later proofs. In Section 3 we prove a variant comparison theorem for $G$-BSDEs. The problem is formulated in detail in Section 4 and we present the technics of approximation via penalization to prove the existence. Furthermore, we state that  the solution by our construction is the largest one using the variant comparison theorem.

\section{Preliminaries}
The main purpose of this section is to recall some basic notions and results of $G$-expectation, which are needed in the sequel. The readers may refer to \cite{HJPS1}, \cite{HJPS2},
\cite{P07a}, \cite{P08a}, \cite{P10} for more details.
\subsection{$G$-expectation}
\begin{definition}
\label{def2.1} Let $\Omega$ be a given set and let $\mathcal{H}$ be a vector
lattice of real valued functions defined on $\Omega$, namely $c\in \mathcal{H%
}$ for each constant $c$ and $|X|\in \mathcal{H}$ if $X\in \mathcal{H}$. $%
\mathcal{H}$ is considered as the space of random variables. A sublinear
expectation $\hat{\mathbb{E}}$ on $\mathcal{H}$ is a functional $
\hat {\mathbb{E}}:\mathcal{H}\rightarrow \mathbb{R}$ satisfying the following
properties: for all $X,Y\in \mathcal{H}$, we have

\begin{description}
\item[(i)] Monotonicity: If $X\geq Y$, then $\hat{\mathbb{E}}[X]\geq
\hat{\mathbb{E}}[Y]$;

\item[(ii)] Constant preserving: $\hat{\mathbb{E}}[c]=c$;

\item[(iii)] Sub-additivity: $\hat{\mathbb{E}}[X+Y]\leq \hat{\mathbb{E}}[X]+%
\hat{\mathbb{E}}[Y]$;

\item[(iv)] Positive homogeneity: $\hat{\mathbb{E}}[\lambda X]=\lambda
\hat{\mathbb{E}}[X]$ for each $\lambda \geq0$.
\end{description}

\end{definition}

\begin{definition}
\label{def2.2} Let $X_{1}$ and $X_{2}$ be two $n$-dimensional random vectors
defined respectively in sublinear expectation spaces $(\Omega_{1},\mathcal{H}%
_{1},\hat{\mathbb{E}}_{1})$ and $(\Omega_{2},\mathcal{H}_{2},\hat{\mathbb{E}}%
_{2})$. They are called identically distributed, denoted by $X_{1}\overset{d}%
{=}X_{2}$, if $\hat{\mathbb{E}}_{1}[\varphi(X_{1})]=\hat{\mathbb{E}}%
_{2}[\varphi(X_{2})]$, for all$\ \varphi \in C_{Lip}(\mathbb{R}^{n})$,
where $C_{Lip}(\mathbb{R}^{n})$ is the space of real valued Lipschitz continuous functions
defined on $\mathbb{R}^{n}$.
\end{definition}

\begin{definition}
\label{def2.3} In a sublinear expectation space $(\Omega,\mathcal{H},%
\hat{\mathbb{E}})$, a random vector $Y=(Y_{1},\cdot \cdot \cdot,Y_{n})$, $Y_{i}\in
\mathcal{H}$, is said to be independent from another random vector $%
X=(X_{1},\cdot \cdot \cdot,X_{m})$, $X_{i}\in \mathcal{H}$ under $%
\hat {\mathbb{E}}[\cdot]$, denoted it by $Y\bot X$, if for every test function $\varphi
\in C_{Lip}(\mathbb{R}^{m}\times \mathbb{R}^{n})$ we have $\hat{\mathbb{E}}%
[\varphi(X,Y)]=\mathbb{\hat{E}}[\hat{\mathbb{E}}[\varphi(x,Y)]_{x=X}]$.
\end{definition}

\begin{definition}
\label{def2.4} ($G$-normal distribution) A $d$-dimensional random vector $%
X=(X_{1},\cdot \cdot \cdot,X_{d})$ in a sublinear expectation space $(\Omega,%
\mathcal{H},\hat{\mathbb{E}})$ is called $G$-normally distributed if for
each $a,b\geq0$ we have
\begin{equation*}
aX+b\bar{X}\overset{d}{=}\sqrt{a^{2}+b^{2}}X,
\end{equation*}
where $\bar{X}$ is an independent copy of $X$, i.e., $\bar{X}\overset{d}{=}X$
and $\bar{X}\bot X$. Here, the letter $G$ denotes the function
\begin{equation*}
G(A):=\frac{1}{2}\hat{\mathbb{E}}[\langle AX,X\rangle]:\mathbb{S}%
_{d}\rightarrow \mathbb{R},
\end{equation*}
where $\mathbb{S}_{d}$ denotes the collection of $d\times d$ symmetric
matrices.
\end{definition}

The function $G(\cdot):\mathbb{S}^d\rightarrow \mathbb{R}$ is a monotonic and sublinear mapping on $\mathbb{S}^d$. In this paper, we suppose that $G$ is non-degenerate, i.e., there exists some $\underline{\sigma}^2>0$ such that $G(A)-G(B)\geq \frac{1}{2}\underline{\sigma}^2 tr[A-B]$ for any $A\geq B$.

Let $\Omega_T=C_{0}([0,T];\mathbb{R}^{d})$, the space of
$\mathbb{R}^{d}$-valued continuous functions on $[0,T]$ with $\omega_{0}=0$, be endowed
with the supremum norm,
and $B=(B^i)_{i=1}^d$ be the canonical
process. Set
\[
L_{ip} (\Omega_T):=\{ \varphi(B_{t_{1}},\cdots,B_{t_{n}}):n\geq1,t_{1}%
,\cdots,t_{n}\in\lbrack0,T],\varphi\in C_{Lip}(\mathbb{R}^{d\times n})\}.
\]

\begin{definition}
For all random variable $X\in L_{ip}(\Omega_T)$ of the following form:
\[\varphi(B_{t_{1}},B_{t_2}-B_{t_1},\cdots,B_{t_{n}}-B_{t_{n-1}}), \ \varphi\in C_{Lip}(\mathbb{R}^{d\times n}),\]
the $G$-expectation is defined as
\[\hat{\mathbb{E}}[\varphi(B_{t_{1}},B_{t_2}-B_{t_1},\cdots,B_{t_{n}}-B_{t_{n-1}})]=\tilde{\mathbb{E}}[\varphi(\sqrt{t_1}\xi_1,\cdots,\sqrt{t_n-t_{n-1}}\xi_n)],\]
where $\xi_1,\cdots,\xi_n$ are identically distributed $d$-dimensional $G$-normally distributed random vectors in a sublinear expectation space $(\tilde{\Omega},\tilde{\mathcal{H}},\mathbb{\tilde{E}})$ such that $\xi_{i+1}$ is independent of $(\xi_1,\cdots,\xi_i)$ for each $i=1,\cdots,n-1$. $(\Omega_T,L_{ip}(\Omega_T),\mathbb{\hat{E}})$ is called  a $G$-expectation space.
The conditional $G$-expectation $\hat{\mathbb{E}}_{t_i}[\cdot]$, $i=1,\cdots,n$, is defined as follows
\[\hat{\mathbb{E}}_{t_i}[\varphi(B_{t_{1}},B_{t_2}-B_{t_1},\cdots,B_{t_{n}}-B_{t_{n-1}})]=\tilde{\varphi}(B_{t_{1}},B_{t_2}-B_{t_1},\cdots,B_{t_{i}}-B_{t_{i-1}}),\]
where
\[\tilde{\varphi}(x_1,\cdots,x_i)=\hat{\mathbb{E}}[\varphi(x_1,\cdots,x_i,B_{t_{i+1}}-B_{t_i},\cdots,B_{t_{n}}-B_{t_{n-1}})].\]
If $t\in(t_i,t_{i+1})$, the conditional $G$-expectation $\hat{\mathbb{E}}_{t}[X]$ could be defined by reformulating $X$ as
\[X=\hat{\varphi}(B_{t_{1}},B_{t_2}-B_{t_1},\cdots,B_t-B_{t_i},B_{t_{i+1}}-B_{t},\cdots,B_{t_{n}}-B_{t_{n-1}}),\ \ \hat{\varphi}\in C_{Lip}(\mathbb{R}^{d\times(n+1)}).\]
\end{definition}


 Denote by $L_{G}^{p}(\Omega_T)$   the completion of
$L_{ip} (\Omega_T)$ under the norm $\Vert\xi\Vert_{L_{G}^{p}}:=(\hat{\mathbb{E}}[|\xi|^{p}])^{1/p}$ for $p\geq1$. It is easy to check that the conditional $G$-expectation is a continuous mapping on $L_{ip}(\Omega_T)$ endowed with the norm $\Vert\cdot\Vert_{L_{G}^{p}}$ and thus can be extended to $L_G^p(\Omega_T)$. Denis et al. \cite{DHP11} proved the following representation theorem of $G$-expectation on $L_G^1(\Omega_T)$.

\begin{theorem}[\cite{DHP11,HP09}]
\label{the1.1}  There exists a weakly compact set
$\mathcal{P}\subset\mathcal{M}_{1}(\Omega_T)$, the set of all probability
measures on $(\Omega_T,\mathcal{B}(\Omega_T))$, such that
\[
\hat{\mathbb{E}}[\xi]=\sup_{P\in\mathcal{P}}E_{P}[\xi]\ \ \text{for
\ all}\ \xi\in  {L}_{G}^{1}{(\Omega_T)}.
\]
$\mathcal{P}$ is called a set that represents $\hat{\mathbb{E}}$.
\end{theorem}

Let $\mathcal{P}$ be a weakly compact set that represents $\hat{\mathbb{E}}$.
For this $\mathcal{P}$, we define capacity%
\[
c(A):=\sup_{P\in\mathcal{P}}P(A),\ A\in\mathcal{B}(\Omega_T).
\]
A set $A\subset\mathcal{B}(\Omega_T)$ is polar if $c(A)=0$.  A
property holds $``quasi$-$surely"$ (q.s.) if it holds outside a
polar set. In the following, we do not distinguish two random variables $X$ and $Y$ if $X=Y$ q.s..

For $\xi\in L_{ip}(\Omega_T)$, let $\mathcal{E}(\xi)=\hat{\mathbb{E}}[\sup_{t\in[0,T]}\hat{\mathbb{E}}_t[\xi]]$, where $\hat{\mathbb{E}}$ is the $G$-expectation. For convenience, we call $\mathcal{E}$ $G$-evaluation. For $p\geq 1$ and $\xi\in L_{ip}(\Omega_T)$, define $\|\xi\|_{p,\mathcal{E}}=[\mathcal{E}(|\xi|^p)]^{1/p}$ and denote by $L_{\mathcal{E}}^p(\Omega_T)$ the completion of $L_{ip}(\Omega_T)$ under $\|\cdot\|_{p,\mathcal{E}}$. We shall give an estimate between the two norms $\|\cdot\|_{L_G^p}$ and $\|\cdot\|_{p,\mathcal{E}}$.


\begin{theorem}[\cite{S11}]\label{the1.2}
For any $\alpha\geq 1$ and $\delta>0$, $L_G^{\alpha+\delta}(\Omega_T)\subset L_{\mathcal{E}}^{\alpha}(\Omega_T)$. More precisely, for any $1<\gamma<\beta:=(\alpha+\delta)/\alpha$, $\gamma\leq 2$, we have
\begin{displaymath}
\|\xi\|_{\alpha,\mathcal{E}}^{\alpha}\leq \gamma^*\{\|\xi\|_{L_G^{\alpha+\delta}}^{\alpha}+14^{1/\gamma}
C_{\beta/\gamma}\|\xi\|_{L_G^{\alpha+\delta}}^{(\alpha+\delta)/\gamma}\},\quad \forall \xi\in L_{ip}(\Omega_T).
\end{displaymath}
where $C_{\beta/\gamma}=\sum_{i=1}^\infty i^{-\beta/\gamma}$,$\gamma^*=\gamma/(\gamma-1)$.
\end{theorem}

\subsection{$G$-It\^{o} calculus}
For simplicity, we only give the definition of $G$-It\^{o}'s integral with respect to 1-dimensional $G$-Brownian motion and its quadratic variation. However, our results in the following sections still hold for the multidimensional case unless otherwise stated.
\begin{definition}
\label{def2.6} Let $M_{G}^{0}(0,T)$ be the collection of processes in the
following form: for a given partition $\{t_{0},\cdot\cdot\cdot,t_{N}\}=\pi
_{T}$ of $[0,T]$,
\begin{equation}\label{123}
\eta_{t}(\omega)=\sum_{j=0}^{N-1}\xi_{j}(\omega)\mathbf{1}_{[t_{j},t_{j+1})}(t),
\end{equation}
where $\xi_{i}\in L_{ip}(\Omega_{t_{i}})$, $i=0,1,2,\cdot\cdot\cdot,N-1$. For each
$p\geq1$ and $\eta\in M_G^0(0,T)$ let $\|\eta\|_{H_G^p}:=\{\hat{\mathbb{E}}[(\int_0^T|\eta_s|^2ds)^{p/2}]\}^{1/p}$, $\Vert\eta\Vert_{M_{G}^{p}}:=(\mathbb{\hat{E}}[\int_{0}^{T}|\eta_{s}|^{p}ds])^{1/p}$ and denote by $H_G^p(0,T)$,  $M_{G}^{p}(0,T)$ the completion
of $M_{G}^{0}(0,T)$ under the norm $\|\cdot\|_{H_G^p}$, $\|\cdot\|_{M_G^p}$ respectively.
\end{definition}

\begin{definition}
For each $\eta\in M_G^0(0,T)$ of the form \eqref{123}, we define the linear mappings $I,L:M_G^0(0,T)\rightarrow L_G^p(\Omega_T)$ as the following:
\begin{align*}
&I(\eta):=\int_0^T\eta_s dB_s=\sum_{j=0}^{N-1}\xi_j(B_{t_{j+1}}-B_{t_j}),\\
&L(\eta):=\int_0^T\eta_s d\langle B\rangle_s=\sum_{j=0}^{N-1}\xi_j(\langle B\rangle_{t_{j+1}}-\langle B\rangle_{t_j}).
\end{align*}
Then $I,L$ can be continuously extended to $H_G^p(0,T)$ and $M_G^p(0,T)$ respectively.
\end{definition}

By Proposition 2.10 in \cite{LP} and classical Burkholder-Davis-Gundy inequality, we have the following estimate for $G$-It\^{o}'s integral.

\begin{proposition}[\cite{HJPS2}]\label{the1.3}
If $\eta\in H_G^{\alpha}(0,T)$ with $\alpha\geq 1$ and $p\in(0,\alpha]$, then we get
$\sup_{u\in[t,T]}|\int_t^u\eta_s dB_s|^p\in L_G^1(\Omega_T)$ and
\begin{displaymath}
\underline{\sigma}^p c_p\hat{\mathbb{E}}_t[(\int_t^T |\eta_s|^2ds)^{p/2}]\leq
\hat{\mathbb{E}}_t[\sup_{u\in[t,T]}|\int_t^u\eta_s dB_s|^p]\leq
\bar{\sigma}^p C_p\hat{\mathbb{E}}_t[(\int_t^T |\eta_s|^2ds)^{p/2}].
\end{displaymath}
\end{proposition}

Let $S_G^0(0,T)=\{h(t,B_{t_1\wedge t}, \ldots,B_{t_n\wedge t}):t_1,\ldots,t_n\in[0,T],h\in C_{b,Lip}(\mathbb{R}^{n+1})\}$. For $p\geq 1$ and $\eta\in S_G^0(0,T)$, set $\|\eta\|_{S_G^p}=\{\hat{\mathbb{E}}[\sup_{t\in[0,T]}|\eta_t|^p]\}^{1/p}$. Denote by $S_G^p(0,T)$ the completion of $S_G^0(0,T)$ under the norm $\|\cdot\|_{S_G^p}$.


We consider the following type of $G$-BSDEs:
\begin{equation}\label{eq1.1}
Y_t=\xi+\int_t^T f(s,Y_s,Z_s)ds+\int_t^T g(s,Y_s,Z_s)d\langle B\rangle_s-\int_t^T Z_s dB_s-(K_T-K_t),
\end{equation}
where
\begin{displaymath}
f(t,\omega,y,z),g(t,\omega,y,z):[0,T]\times\Omega_T\times\mathbb{R}\times\mathbb{R}\rightarrow \mathbb{R},
\end{displaymath}
satisfy the following properties:
\begin{description}
\item[(H1')] There exists some $\beta>1$ such that for any $y,z$, $f(\cdot,\cdot,y,z),g(\cdot,\cdot,y,z)\in M_G^{\beta}(0,T)$,
\item[(H2)] There exists some $L>0$ such that
\begin{displaymath}
|f(t,y,z)-f(t,y',z')|+|g(t,y,z)-g(t,y',z')|\leq L(|y-y'|+|z-z'|).
\end{displaymath}
\end{description}

For simplicity, we denote by $\mathfrak{S}_G^{\alpha}(0,T)$ the collection of process $(Y,Z,K)$ such that $Y\in S_G^{\alpha}(0,T)$, $Z\in H_G^{\alpha}(0,T)$, $K$ is a decreasing $G$-martingale with $K_0=0$ and $K_T\in L_G^{\alpha}(\Omega_T)$.

\begin{definition}
Let $\xi\in L_G^{\beta}(\Omega_T)$ and $f$, $g$ satisfies (H1') and (H2) for some $\beta>1$. A triplet of processes $(Y,Z,K)$ is called a solution of equation \eqref{eq1.1} if for some $1<\alpha\leq \beta$ the following properties hold:
\begin{description}
\item[(a)]$(Y,Z,K)\in\mathfrak{S}_G^{\alpha}(0,T)$;
\item[(b)]$Y_t=\xi+\int_t^T f(s,Y_s,Z_s)ds+\int_t^T g(s,Y_s,Z_s)d\langle B\rangle_s-\int_t^T Z_s dB_s-(K_T-K_t)$.
\end{description}
\end{definition}

\begin{theorem}[\cite{HJPS1}]\label{the1.4}
Assume that $\xi\in L_G^{\beta}(\Omega_T)$ and $f,g$ satisfy (H1') and (H2) for some $\beta>1$. Then equation \eqref{eq1.1} has a unique solution $(Y,Z,K)$. Moreover, for any $1<\alpha<\beta$, we have $Z\in H_G^{\alpha}(0,T)$, $K_T\in L_G^{\alpha}(\Omega_T)$ and
\begin{displaymath}
|Y_t|^\alpha\leq C\hat{\mathbb{E}}_t[|\xi|^\alpha+\int_t^T |f(s,0,0)|^\alpha+|g(s,0,0)|^\alpha ds],
\end{displaymath}
where the constant $C$ depends on $\alpha$, $T$, $\underline{\sigma}$ and $L$.
\end{theorem}

\begin{theorem}[\cite{HJPS2}]\label{the1.5}
Let $(Y_t^l,Z_t^l,K_t^l)_{t\leq T}$, $l=1,2$, be the solutions of the following $G$-BSDEs:
\begin{displaymath}
Y^l_t=\xi^l+\int_t^T f^l(s,Y^l_s,Z^l_s)ds+\int_t^T g^l(s,Y^l_s,Z^l_s)d\langle B\rangle_s+V_T^l-V_t^l-\int_t^T Z^l_s dB_s-(K^l_T-K^l_t),
\end{displaymath}
where $\{V_t^l\}_{0\leq t\leq T}$ are RCLL processes such that $\hat{\mathbb{E}}[\sup_{t\in[0,T]}|V_t^l|^\beta]<\infty$, $f^l,\ g^l$ satisfy (H1') and (H2), $\xi^l\in L_G^{\beta}(\Omega_T)$ with $\beta>1$. If $\xi^1\geq \xi^2$, $f^1\geq f^2$, $g^1\geq g^2$, for $i,j=1,\cdots,d$, $V_t^1-V_t^2$ is an increasing process, then $Y_t^1\geq Y_t^2$.
\end{theorem}

\section{A variant comparison theorem}
In this section, we introduce a variant comparison theorem for solutions to $G$-BSDEs. First, we state some basic properties as preliminaries.

\begin{lemma}\label{th10}
Let $X_t\in S^{\alpha}_G(0,T)$, where $\alpha>1$. Set $X_t^n=\sum_{i=0}^{n-1}X_{t_i^n}I_{[t_i^n,t_{i+1}^n)}(t)$, where $t_i^n=\frac{iT}{n}$, $i=0,\cdots,n$,
 $1/\alpha+1/\alpha^*=1$. Suppose that $K$ is a $G$-submartingale with finite variation satisfying $K_0=0$ and $\hat{\mathbb{E}}[|Var(K)|^{\alpha*}]<\infty$, where $Var(K)$ is the total variation of $K$ on $[0,T]$, then
\begin{displaymath}
\hat{\mathbb{E}}[\sup_{t\in[0,T]}|\int_0^t(X_s^n-X_s)dK_s|]\rightarrow0.
\end{displaymath}
\end{lemma}
\begin{proof}
It is easy to check that
\begin{displaymath}
\sup_{t\in[0,T]}|\int_0^t(X_s^n-X_s)dK_s|\leq \sup_{t\in[0,T]}|X_t^n-X_t||Var(K)|.
\end{displaymath}
By applying Lemma 3.2 in \cite{HJPS1}, we have
\begin{displaymath}\hat{\mathbb{E}}[\sup_{t\in[0,T]}|\int_0^t(X_s^n-X_s)dK_s|]\leq \|\sup_{t\in[0,T]}|X_t^n-X_t|\|_{L^\alpha_G}\|Var(K)\|_{L^{\alpha^*}_G}\rightarrow 0.
\end{displaymath}
\end{proof}

\begin{lemma}\label{th11}
Let $X_t\in S^{\alpha}_G(0,T)$, where $\alpha>1$,
$1/\alpha+1/\alpha^*=1$. Suppose that $K^j$ is a $G$-submartingale with finite variation satisfying $K^j_0=0$ and $\hat{\mathbb{E}}[|Var(K^j)|^{\alpha*}]<\infty$, $j=1,2$, then
\begin{displaymath}
\int_0^t X_s^+ dK_s^1+\int_0^t X_s^- dK_s^2,
\end{displaymath}
is a $G$-submartingale.
\end{lemma}
\begin{proof}
It suffices to prove that the process $\int_0^t (X_s^n)^+ dK_s^1+\int_0^t (X_s^n)^- dK_s^2$ is a $G$-submartingale, where $X^n$ is the same as Lemma \ref{th10}. Then for any
$t\in[t^n_i,t^n_{i+1})$,
\begin{displaymath}
\begin{split}
&\hat{\mathbb{E}}_t[X^+_{t_i^n}(K^1_{t_{i+1}^n}-K^1_{t_i^n})+X^-_{t_i^n}(K^2_{t_{i+1}^n}-K^2_{t_i^n})]\\
&=X^+_{t_i^n}\hat{\mathbb{E}}_t[(K^1_{t_{i+1}^n}-K^1_{t_i^n})]+X^-_{t_i^n}\hat{\mathbb{E}}_t[(K^2_{t_{i+1}^n}-K^2_{t_i^n})]\\
&\geq X^+_{t_i^n}(K^1_t-K^1_{t_i^n})+X^-_{t_i^n}(K^2_t-K^2_{t_i^n}).
\end{split}
\end{displaymath}
From this we have the desired result.
\end{proof}

Consider the following equation
\begin{equation}\label{4}
Y_t=\xi+\int_t^T f_s ds+\int_t^T g_s d\langle B\rangle_s-\int_t^T Z_s dB_s-(K_T-K_t),
\end{equation}
where $f_s=a_sY_s+b_sZ_s+m_s$, $g_s=c_sY_s+d_sZ_s+n_s$, $K$ is given such that it is a $G$-submartingale with finite variation satisfying $K_0=0$. It is worth pointing out that there may not exist a pair $(Y,Z)$ satisfying \eqref{4}. If there does exist a solution of equation \eqref{4}, to solve this problem, first we construct an auxiliary extended $\tilde{G}$-expectation space $(\tilde{\Omega}_T,L_{\tilde{G}}^1(\tilde{\Omega}_T),\hat{\mathbb{E}}^{\tilde{G}})$ with $\tilde{\Omega}_T=C_0([0,T],\mathbb{R}^2)$ and

\begin{displaymath}
\tilde{G}(A)=\frac{1}{2}\sup_{\underline{\sigma}^2\leq v\leq \overline{\sigma}^2} tr\Big[A \begin{bmatrix}v & 1 \\ 1 &v^{-1}\end{bmatrix}\Big],
A\in\mathbb{S}^2.
\end{displaymath}
Let $\{(B_t,\tilde{B}_t)\}$ be the canonical process in the extended space.

\begin{remark}
It is easy to check that $\langle B,\tilde{B}\rangle_t=t$. In particular, if $\underline{\sigma}^2=\overline{\sigma}^2$, we can further get $\tilde{B}_t=\overline{\sigma}^{-2}B_t$.
\end{remark}

Now we consider the following $\tilde{G}$-SDE:
\begin{displaymath}
X_t=1+\int_0^t a_sX_s ds+\int_0^t c_sX_s d\langle B\rangle_s+\int_0^t d_sX_s dB_s
+\int_0^t b_sX_s d\tilde{B}_s.
\end{displaymath}
We may solve it explicitly and get that
\[X_t=\exp(\int_0^t (a_s-b_sd_s)ds+\int_0^t c_sd\langle B\rangle_s)\mathcal{E}_t^B\mathcal{E}_t^{\tilde{B}},\]
where $\mathcal{E}_t^B=\exp(\int_0^t d_s dB_s-\frac{1}{2}\int_0^t d_s^2d\langle B\rangle_s)$, $\mathcal{E}_t^{\tilde{B}}=\exp(\int_0^t b_s d\tilde{B}_s-\frac{1}{2}\int_0^t b_s^2d\langle \tilde{B}\rangle_s)$. Then applying $G$-It\^{o}'s formula to $X_tY_t$, we derive that
\begin{displaymath}
\begin{split}
&X_tY_t+\int_t^T(X_sZ_s+d_sX_sY_s)dB_s+\int_t^T b_sX_sY_s d\tilde{B}_s+\int_t^T X_sdK_s\\
&=X_T\xi+\int_t^T m_sX_s ds+\int_t^T n_sX_s d\langle B\rangle_s.
\end{split}
\end{displaymath}
From Lemma \ref{th11}, we have $\hat{\mathbb{E}}^{\tilde{G}}_t[\int_t^T X_sdK_s]\geq 0$. Taking $\hat{\mathbb{E}}_t^{\tilde{G}}$ conditional expectations on both sides of the above equality, it follows that
\begin{displaymath}
Y_t\leq (X_t)^{-1}\hat{\mathbb{E}}^{\tilde{G}}_t[X_T\xi+\int_t^T m_sX_s ds+\int_t^T n_sX_s d\langle B\rangle_s].
\end{displaymath}

Consider the following equation
\begin{displaymath}
Y_t=K_T+\int_t^T (a_sY_s+b_sZ_s-a_sK_s) ds+\int_t^T (c_sY_s+d_sZ_s-c_sK_s) d\langle B\rangle_s-\int_t^T Z_s dB_s-(K_T-K_t).
\end{displaymath}
 Here, $K$ is a given $G$-submartingale with finite variation satisfying $K_0=0$. It is easy to check that $Y_t=K_t$, $Z_t=0$ is the solution of the above equation. Applying the analysis above, we have
\begin{equation}\label{eq5}
K_t=Y_t\leq (X_t)^{-1}\hat{\mathbb{E}}^{\tilde{G}}_t[X_T K_T-\int_t^T a_sK_sX_s ds-\int_t^T c_sK_sX_s d\langle B\rangle_s].
\end{equation}

\begin{remark}If $K$ in \eqref{4} is a decreasing $G$-martingale, then the two sides of the above inequalities are equal. \end{remark}

\begin{theorem}\label{th12}
Assume that $\xi^i\in L_G^\beta(\Omega_T)$, $f_i,g_i$ satisfy (H1') and (H2) in Section 2 with $\beta>1$, $i=1,2$. Let $(Y^2_t,Z^2_t,K^2_t)$ be the solution of $G$-BSDE with generators $f_2,g_2$ and terminal value $\xi^2$, $(Y_t^1,Z_t^1)$ satisfy the following equation
\begin{displaymath}
Y^1_t=\xi^1+\int_t^T f_1(s,Y^1_s,Z^1_s)ds+\int_t^T g_1(s,Y^1_s,Z^1_s)d\langle B\rangle_s
-\int_t^T Z_s^1 dB_s-(K^1_T-K^1_t),
\end{displaymath}
where $K^1$ is a $G$-submartingale with finite variation satisfying $K_0=0$ and $\hat{\mathbb{E}}[|Var(K^1)|^\beta]<\infty$. If $\xi^1\leq \xi^2$, $f_1\leq f_2$, $g_1\leq g_2$, then $Y_t^1\leq Y_t^2$.
\end{theorem}

\begin{proof}
Let $\hat{Y}_t=Y_t^2-Y_t^1$, $\hat{Z}_t=Z_t^2-Z^1_t$, $\hat{f}_s=f_2(s,Y_s^2,Z_s^2)-f_1(s,Y_s^1,Z_s^1)$,
$\hat{g}_s=g_2(s,Y_s^2,Z_s^2)-g_1(s,Y_s^1,Z_s^1)$, $\hat{\xi}=\xi^2-\xi^1$. Then we have
\begin{equation}\label{eq6}
\hat{Y}_t+K_t^1=\hat{\xi}+K_T^1+\int_t^T \hat{f}_s ds +\int_t^T \hat{g}_s d\langle B\rangle_s
-\int_t^T \hat{Z}_s dB_s-(K_T^2-K_t^2).
\end{equation}

For each fixed $\varepsilon>0$, by the proof of Theorem 3.6 in \cite{HJPS2}, we can get 
\begin{displaymath}
\hat{f}_s=a^\varepsilon_s\hat{Y}_s+b^\varepsilon_s\hat{Z}_s+m_s-m_s^\varepsilon,
\hat{g}_s=c^\varepsilon_s\hat{Y}_s+d^\varepsilon_s\hat{Z}_s+n_s-n_s^\varepsilon,
\end{displaymath}
where $|m^\varepsilon_s|\leq 4L\varepsilon$, $|n^\varepsilon_s|\leq 4L\varepsilon$,
$m_s=f_2(s,Y_s^1,Z_s^1)-f_1(s,Y_s^1,Z_s^1)\geq 0$, $n_s=g_2(s,Y_s^1,Z_s^1)-g_1(s,Y_s^1,Z_s^1)\geq 0$, $\psi^\varepsilon\in M_G^2(0,T)$ and $|\psi^\varepsilon|\leq L$ for $\psi=a$, $b$, $c$, $d$.

Recalling \eqref{eq5}, we can solve \eqref{eq6} to get
\begin{displaymath}
\begin{split}%
\hat{Y}_t+K_t^1
=&(X_t^\varepsilon)^{-1}\hat{\mathbb{E}}_t^{\tilde{G}}[X_T^\varepsilon(\hat{\xi}+K_T^1)
+\int_t^T(\widetilde{m}_s^\varepsilon-a_s^\varepsilon K_s^1)X_s^\varepsilon ds
+\int_t^T(\widetilde{n}_s^\varepsilon-c_s^\varepsilon K_s^1)X_s^\varepsilon d\langle B\rangle_s]\\
\geq&(X_t^\varepsilon)^{-1}\hat{\mathbb{E}}_t^{\tilde{G}}[X_T^\varepsilon K_T^1
+\int_t^T(-m_s^\varepsilon-a_s^\varepsilon K_s^1)X_s^\varepsilon ds
+\int_t^T(-n_s^\varepsilon-c_s^\varepsilon K_s^1)X_s^\varepsilon d\langle B\rangle_s]\\
\geq &(X_t^\varepsilon)^{-1}\hat{\mathbb{E}}_t^{\tilde{G}}[X_T^\varepsilon K_T^1
-\int_t^T a_s^\varepsilon K_s^1X_s^\varepsilon ds
-\int_t^T c_s^\varepsilon K_s^1X_s^\varepsilon d\langle B\rangle_s]\\
&-(X_t^\varepsilon)^{-1}\hat{\mathbb{E}}_t^{\tilde{G}}[\int_t^T m_s^\varepsilon X_s^\varepsilon ds+
\int_t^T n_s^\varepsilon X_s^\varepsilon d\langle B\rangle_s]\\
\geq &K_t^1-4L\varepsilon (X_t^\varepsilon)^{-1}\hat{\mathbb{E}}_t^{\tilde{G}}[\int_t^T|X_s^\varepsilon|ds+
\int_t^T |X_s^\varepsilon| d\langle B\rangle_s],
\end{split}
\end{displaymath}
where $\widetilde{m}_s^\varepsilon=m_s-m_s^\varepsilon$, $\widetilde{n}_s^\varepsilon=n_s-n_s^\varepsilon$, $\{X_t^\varepsilon\}_{t\in[0,T]}$ is the solution of the following equation
\begin{displaymath}
X_t^\varepsilon=1+\int_0^t a_s^\varepsilon X^\varepsilon_s ds+\int_0^t c^\varepsilon_sX_s^\varepsilon d\langle B\rangle_s+
\int_0^t d_s^\varepsilon X^\varepsilon_s dB_s
+\int_0^t b_s^\varepsilon X_s^\varepsilon d\tilde{B}_s.
\end{displaymath}
Then by letting $\varepsilon\rightarrow 0$, we can derive the desired result.
\end{proof}

\section{Reflected $G$-BSDE with an upper obstacle}
El Karoui, Kapoudjian, Pardoux, Peng and Quenez \cite{KKPPQ} introduced the reflected BSDE with a lower obstacle. An additional increasing process should be added in this equation to keep the solution be above the given obstacle. Substitute a decreasing process for the increasing one, we can use the same method to deal with the reflected BSDE with an upper obstacle.  However, under the $G$-framework, due to the appearance of the decreasing $G$-martingale in the penalized $G$-BSDEs, these two cases are significantly different. Now we reformulate this problems as follows.

We are given these parameters: the generator $f$ and $g$, the obstacle process $\{S_t\}_{t\in[0,T]}$ and the terminal value $\xi$, where $f$ and $g$ are maps
\begin{displaymath}
f(t,\omega,y,z),g(t,\omega,y,z):[0,T]\times\Omega_T\times\mathbb{R}^2\rightarrow\mathbb{R}.
\end{displaymath}

The following assumptions will be needed throughout this section. There exists some $\beta>2$ such that
\begin{description}
\item[(A1)] for any $y,z$, $f(\cdot,\cdot,y,z)$, $g(\cdot,\cdot,y,z)\in M_G^\beta(0,T)$ and  $\hat{\mathbb{E}}[\sup_{t\in[0,T]}(|f(t,0,0)|^\beta+|g(t,0,0)|^\beta)]<\infty$;
\item[(A2)] $|f(t,\omega,y,z)-f(t,\omega,y',z')|+|g(t,\omega,y,z)-g(t,\omega,y',z')|\leq L(|y-y'|+|z-z'|)$ for some $L>0$;
\item[(A3)] $\{S_t\}_{t\in[0,T]}\in S_G^\beta(0,T)$ is of the following form
\begin{equation}\label{S}
S_t=S_0+\int_0^t b(s)ds+\int_0^t l(s)d\langle B\rangle_s+\int_0^t \sigma(s)dB_s,
\end{equation}
where $\{b(t)\}_{t\in[0,T]}$, $\{l(t)\}_{t\in[0,T]}$ belong to $M_G^\beta(0,T)$ and $\{\sigma(t)\}_{t\in[0,T]}$ belongs to $H_G^\beta(0,T)$. Furthermore,  $\hat{\mathbb{E}}[\sup_{t\in[0,T]}\{|b(t)|^\beta+|l(t)|^\beta+|\sigma(t)|^\beta\}]<\infty$;
\item[(A4)] $\xi\in L_G^\beta(\Omega_T)$ and $\xi\leq S_T$, $q.s.$.
\end{description}

Then we can introduce our reflected $G$-BSDE with an upper obstacle. A triplet of processes $(Y,Z,A)$ is called a solution of reflected $G$-BSDE if for some $1<\alpha\leq \beta$ the following properties are satisfied:
\begin{description}
\item[(i)]$(Y,Z,A)\in\mathbb{S}_G^{\alpha}(0,T)$ and $Y_t\leq S_t$, $0\leq t\leq T$;
\item[(ii)]$Y_t=\xi+\int_t^T f(s,Y_s,Z_s)ds+\int_t^T g(s,Y_s,Z_s)d\langle B\rangle_s
-\int_t^T Z_s dB_s+(A_T-A_t)$;
\item[(iii)]$\{-\int_0^t (S_s-Y_s)dA_s\}_{t\in[0,T]}$ is a decreasing $G$-martingale.
\end{description}
Here we denote by $\mathbb{S}_G^{\alpha}(0,T)$ the collection of process $(Y,Z,A)$ such that $Y\in S_G^{\alpha}(0,T)$, $Z\in H_G^{\alpha}(0,T)$, $A\in S_G^\alpha(0,T)$ is a continuous process  with finite variation satisfying $A_0=0$ and $-A$ is a $G$-submartingale.

For notational simplification, in this paper we only consider the case $g\equiv 0$ and $l\equiv 0$. But the results still hold for the other cases.

 \begin{theorem}\label{th1}
Under the above assumptions, in particular (A1)-(A4), the reflected $G$-BSDE with parameters $(\xi,f,S)$ has a solution $(Y,Z,A)$. This solution is the maximal one in the sense that, if $(Y',Z',A')$ is another solution, then $Y_t\geq Y'_t$, for all $t\in[0,T]$
\end{theorem}

The proof will be divided into a sequence of lemmas. For $f$, $\{S_t\}_{t\in[0,T]}$ and $\xi$ satisfy (A1)-(A4) with $\beta>2$. We now consider the following family of $G$-BSDEs parameterized by $n=1,2,\ldots$.
\begin{equation}\label{eq1}
Y_t^n=\xi+\int_t^T f(s,Y_s^n,Z^n_s)ds-n\int_t^T(Y_s^n-S_s)^+ds-\int_t^T Z_s^ndB_s-(K_T^n-K_t^n).
\end{equation}

Now let $L_t^n=-n\int_0^t (Y_s^n-S_s)^+ds$. Then $(L_t^n)_{t\in[0,T]}$ is a nonincreasing process. We can rewrite reflected $G$-BSDE \eqref{eq1} as
\begin{equation}
Y_t^n=\xi+\int_t^T f(s,Y_s^n,Z^n_s)ds-\int_t^T Z_s^ndB_s-(K_T^n-K_t^n)+(L_T^n-L_t^n).
\end{equation}

\begin{lemma}\label{th2}
There exists a constant $C$ independent of $n$, such that for $1<\alpha<\beta$, we have \[\hat{\mathbb{E}}[\sup_{t\in[0,T]}|Y_t^n|^\alpha]\leq C.\]
\end{lemma}

\begin{proof}
For simplicity, first we consider the case where $S\equiv 0$. For the case that $S$ is a $G$-It\^{o} process, we may refer to Remark \ref{th}. For any $r,\varepsilon>0$, set $\tilde{Y}_t=(Y_t^n)^2+\varepsilon_\alpha$,
where $\varepsilon_\alpha=\varepsilon(1-\alpha/2)^+$. Applying It\^{o}'s formula to $\tilde{Y}_t^{\alpha/2}e^{rt}$  yields that
\begin{displaymath}
\begin{split}
&\quad\tilde{Y}_t^{\alpha/2}e^{rt}+\int_t^T re^{rs}\tilde{Y}_s^{\alpha/2}ds+\int_t^T \frac{\alpha}{2} e^{rs}
\tilde{Y}_s^{\alpha/2-1}(Z_s^n)^2d\langle B\rangle_s\\
&=(|\xi|^2+\varepsilon_\alpha)^{\frac{\alpha}{2}}e^{rT}+
\alpha(1-\frac{\alpha}{2})\int_t^Te^{rs}\tilde{Y}_s^{\alpha/2-2}(Y_s^n)^2(Z_s^n)^2d\langle B\rangle_s
+\int_t^T\alpha e^{rs}\tilde{Y}_s^{\alpha/2-1}Y_s^ndL_s^n\\
&\quad+\int_t^T{\alpha} e^{rs}\tilde{Y}_s^{\alpha/2-1}Y_s^nf(s,Y_s^n,Z_s^n)ds-
\int_t^T\alpha e^{rs}\tilde{Y}_s^{\alpha/2-1}(Y_s^nZ_s^ndB_s+Y_s^ndK_s^n)\\
&\leq(|\xi|^2+\varepsilon_\alpha)^{\frac{\alpha}{2}}e^{rT}+
\alpha(1-\frac{\alpha}{2})\int_t^Te^{rs}\tilde{Y}_s^{\alpha/2-2}(Y_s^n)^2(Z_s^n)^2d\langle B\rangle_s\\
&\quad+\int_t^T{\alpha} e^{rs}\tilde{Y}_s^{\alpha/2-1/2}|f(s,Y_s^n,Z_s^n)|ds-(M_T-M_t),
\end{split}
\end{displaymath}
where $M_t=\int_0^t\alpha e^{rs}\tilde{Y}_s^{\alpha/2-1}(Y_s^nZ_s^ndB_s+(Y_s^n)^+dK_s^n)$ is a $G$-martingale. In the last inequality, we use the fact that $-y(y)^+\leq 0$ for any $y\in\mathbb{R}$. From the assumption of $f$ and the Young inequality, we have
\begin{displaymath}
\begin{split}
\int_t^T{\alpha} e^{rs}\tilde{Y}_s^{\frac{\alpha-1}{2}}|f(s,Y_s^n,Z_s^n)|ds
\leq &\int_t^T \alpha e^{rs}\tilde{Y}_s^{\frac{\alpha-1}{2}} [|f(s,0,0)|+L|Y^n_s|+L|Z^n_s|]ds\\
\leq &\int_t^T e^{rs}|f(s,0,0)|^\alpha ds+\frac{\alpha(\alpha-1)}{4}\int_t^Te^{rs}\tilde{Y}_s^{\alpha/2-1}(Z_s^n)^2d\langle B\rangle_s\\
&+(\alpha-1+\alpha L+\frac{\alpha L^2}{\underline{\sigma}^2(\alpha-1)})\int_t^T e^{rs}\tilde{Y}_s^{\alpha/2}ds.
\end{split}
\end{displaymath}
Setting $r=\alpha+\alpha L+\frac{\alpha L^2}{\underline{\sigma}^2(\alpha-1)}$, we can get
\begin{displaymath}
\tilde{Y}_t^{\alpha/2}e^{rt}+M_T-M_t\leq (|\xi|^2+\varepsilon_\alpha)^{\frac{\alpha}{2}}e^{rT}+\int_t^T e^{rs}|f(s,0,0)|^\alpha ds,
\end{displaymath}
Taking conditional expectations on both sides and letting $\varepsilon\rightarrow 0$, we have
\begin{displaymath}
|Y_t^n|^\alpha\leq C\hat{\mathbb{E}}_t[|\xi|^\alpha+\int_t^T|f(s,0,0)|^\alpha ds].
\end{displaymath}
By Theorem \ref{the1.2}, we can conclude that for $1<\alpha<\beta$, there exists a constant $C$ independent of $n$ such that $\hat{\mathbb{E}}[\sup_{t\in[0,T]}|Y_t^n|^\alpha]\leq C$.
\end{proof}

\begin{remark}\label{th}
For the case when the obstacle process  $S$ is given in \eqref{S},
 let $\widetilde{Y}^n_t={Y}_t^n-S_t$ and $\widetilde{Z}_t^n={Z}_t^n-\sigma(s)$. We can rewrite \eqref{eq1} as the following:
\begin{displaymath}
\widetilde{Y}_t^n=\xi-S_T+\int_t^T f(s,\widetilde{Y}_s^n+S_s,\widetilde{Z}_s^n+\sigma(s))+b(s)ds-n\int_t^T(\widetilde{Y}_s^n)^+ds-\int_t^T \widetilde{Z}_s^ndB_s-({K}_T^n-{K}_t^n).
\end{displaymath}
Using the same method as the proof of Lemma \ref{th2}, we get that
\begin{displaymath}
|\widetilde{Y}_t^n|^\alpha\leq C\hat{\mathbb{E}}_t[|\xi-S_T|^\alpha+\int_t^T|f(s,S_s,\sigma(s))+b(s)|^\alpha ds].
\end{displaymath}
Thus, we conclude that, for $1<\alpha<\beta$, there exists a constant $C$ independent of $n$ such that $\hat{\mathbb{E}}[\sup_{t\in[0,T]}|{Y}_t^n|^\alpha]\leq C$.
\end{remark}

Compared with Lemma 4.3 in \cite{lp}, the following result is sharper. More importantly, this lemma allows us to establish uniform estimates on the sequence $(K^n,L^n)$ and then to obtain the convergence of $(Y^n)$. We apply a nonlinear Girsanov transformation approach to prove this result. First, we consider the following $G$-BSDE driven by 1-dimensional $G$-Brownian motion:
\begin{displaymath}
Y_t^L=\xi+\int_t^T L|Z_s^L|ds-\int_t^T Z_s^LdB_s-(K_T^L-K_t^L).
\end{displaymath}

For each $\xi\in L_G^\beta(\Omega_T)$ with $\beta>1$, we define
\begin{displaymath}
\tilde{\mathbb{E}}^L_t[\xi]:=Y_t^L.
\end{displaymath}

By Theorem 5.1 in \cite{HJPS2}, $\tilde{\mathbb{E}}^L_t[\cdot]$ is a consistent sublinear expectation.
\begin{lemma}\label{th3}
There exists a constant $C$ independent of $n$ such that for $1<\alpha<\beta$,
\begin{displaymath}
\hat{\mathbb{E}}[\sup_{t\in[0,T]}|(Y_t^n-S_t)^+|^\alpha]\leq \frac{C}{n^\alpha}.
\end{displaymath}
\end{lemma}

\begin{proof}
Set $\tilde{Y}_t^n=Y_t^n-S_t$, $\tilde{Z}_t^n=Z_t^n-\sigma(t)$, we can rewrite $G$-BSDE \eqref{eq1} as
\begin{equation}\label{eq2}
\tilde{Y}_t^n=\xi-S_T+\int_t^T [f(s,Y_s^n,Z_s^n)+b(s)]ds-\int_t^T n(\tilde{Y}_s^n)^+ds-\int_t^T \tilde{Z}_s^ndB_s-(K_T^n-K_t^n).
\end{equation}

For each given $\varepsilon>0$, we can choose a Lipschitz continuous function $l(\cdot)$ such that $I_{[-\varepsilon,\varepsilon]}(x)\leq l(x)\leq I_{[-2\varepsilon,2\varepsilon]}(x)$. Thus we have
\begin{displaymath}
f(s,Y_s^n,Z_s^n)-f(s,Y_s^n,0)=(f(s,Y_s^n,Z_s^n)-f(s,Y_s^n,0))l(Z_s^n)+a^{\varepsilon,n}_sZ_s^n=:m_s^{\varepsilon,n}+a^{\varepsilon,n}_sZ_s^n,
\end{displaymath}
where
\begin{displaymath}
a^{\varepsilon,n}_s=\begin{cases}(1-l(Z_s^n))(f(s,Y_s^n,Z_s^n)-f(s,Y_s^n,0))(Z_s^n)^{-1}, &\textrm{ if }Z_s^n\neq 0;\\
0, &\textrm{ otherwise }.
\end{cases}
\end{displaymath}
It is easy to check that $a^{\varepsilon,n}\in M_G^2(0,T)$, $|a^{\varepsilon,n}_s|\leq L$ and $|m_s^{\varepsilon,n}|\leq 2L\varepsilon$. Then we can get
\begin{displaymath}
f(s,Y_s^n,Z_s^n)=f(s,Y_s^n,0)+a^{\varepsilon,n}_sZ_s^n+m_s^{\varepsilon,n}=f(s,Y_s^n,0)+a^{\varepsilon,n}_s\sigma(s)+a^{\varepsilon,n}_s\tilde{Z}_s^n+m_s^{\varepsilon,n}.
\end{displaymath}

Now we consider the following $G$-BSDE:
\begin{displaymath}
Y_t^{\varepsilon,n}=\xi+\int_t^T a^{\varepsilon,n}_sZ_s^{\varepsilon,n}ds-\int_t^T Z_s^{\varepsilon,n}dB_s-(K_T^{\varepsilon,n}-K_t^{\varepsilon,n}).
\end{displaymath}
For each $\xi\in L_G^\beta(\Omega_T)$ with $\beta>1$, we define
\begin{displaymath}
\tilde{\mathbb{E}}^{\varepsilon,n}_t[\xi]:=Y_t^{\varepsilon,n}.
\end{displaymath}
Set $\tilde{B}^{\varepsilon,n}_t=B_t-\int_0^t a^{\varepsilon,n}_s ds$. By Theorem 5.2 in \cite{HJPS2}, $\{\tilde{B}^{\varepsilon,n}_t\}$ is a $G$-Brownian motion under $\tilde{\mathbb{E}}^{\varepsilon,n}[\cdot]$. Moreover, by Theorem \ref{the1.5}, we have $\tilde{\mathbb{E}}^{\varepsilon,n}_t[\xi]\leq \tilde{\mathbb{E}}^L_t[\xi]$, $\forall \xi\in L_G^\beta(\Omega_T)$. We can rewrite $G$-BSDE \eqref{eq2} as the following
\begin{displaymath}
\tilde{Y}_t^n=\xi-S_T+\int_t^T f^{\varepsilon,n}(s)ds-\int_t^T n(\tilde{Y}_s^n)^+ds-\int_t^T \tilde{Z}_s^nd\tilde{B}^{n,\varepsilon}_s-(K_T^n-K_t^n),
\end{displaymath}
where $f^{\varepsilon,n}(s)=f(s,Y_s^n,0)+m_s^{\varepsilon,n}+a_s^{\varepsilon,n}\sigma(s)+b(s)$. Applying $G$-It\^{o}'s formula to $e^{-nt}\tilde{Y}_t^n$, we get
\begin{align*}
\tilde{Y}_t^n+\int_t^T e^{n(t-s)}dK_s^n=&(\xi-S_T)e^{n(t-T)}+\int_t^T ne^{n(t-s)}[\tilde{Y}_s^n-(\tilde{Y}_s^n)^+]ds\\
&+\int_t^T e^{n(t-s)}f^{\varepsilon,n}(s)ds-\int_t^T e^{n(t-s)}\tilde{Z}_s^ndB^{\varepsilon,n}_s\\
\leq &\int_t^T e^{n(t-s)}|f^{\varepsilon,n}(s)|ds-\int_t^T e^{n(t-s)}\tilde{Z}_s^ndB^{\varepsilon,n}_s.
\end{align*}
Note that $\tilde{\mathbb{E}}^{\varepsilon,n}_s[K_t^n]=K_s^n$ for any $0\leq s\leq t\leq T$ by Theorem 5.1 in \cite{HJPS2}. Taking $\tilde{\mathbb{E}}^{\varepsilon,n}_t$ conditional expectation on both sides, we have
\begin{align*}
\tilde{Y}_t^n\leq &\tilde{\mathbb{E}}^{\varepsilon,n}_t[\int_t^T e^{n(t-s)}|f^{\varepsilon,n}(s)|ds]
\leq \tilde{\mathbb{E}}^L_t[\int_t^T e^{n(t-s)}|f^{\varepsilon,n}(s)|ds]\\
\leq &\tilde{\mathbb{E}}^L_t[\int_t^T e^{n(t-s)}\sup_{u\in[0,T]}[|f(u,0,0)|+L|Y_u^n|+|m_u^{\varepsilon,n}|+L|\sigma(u)|+|b(u)|]ds]\\
\leq &\frac{C}{n}\tilde{\mathbb{E}}^L_t[\sup_{u\in[0,T]}[|f(u,0,0)|+|Y_u^n|+|\sigma(u)|+|b(u)|]+\varepsilon].
\end{align*}
By Theorem \ref{the1.4}, for $1<\alpha<\beta$, it follows that
\begin{align*}
|(\tilde{Y}_t^n)^+|^\alpha\leq &\frac{C}{n^\alpha}(\tilde{\mathbb{E}}^L_t[\sup_{u\in[0,T]}[|f(u,0,0)|+|Y_u^n|+|\sigma(u)|+|b(u)|]+\varepsilon])^\alpha\\
\leq &\frac{C}{n^\alpha}\hat{\mathbb{E}}_t[\sup_{u\in[0,T]}[|f(u,0,0)|+|Y_u^n|+|\sigma(u)|+|b(u)|+\varepsilon]^\alpha].
\end{align*}
Then applying Lemma \ref{th2} and Theorem \ref{the1.2}, letting $\varepsilon\rightarrow\infty$, we get the desired result.
\end{proof}

\begin{lemma}\label{th4}
There exists a constant $C$ independent of $n$, such that for $1<\alpha<\beta$,
\begin{displaymath}
\hat{\mathbb{E}}[|L_T^n|^\alpha]=\hat{\mathbb{E}}[n^\alpha(\int_0^T (Y_s^n-S_s)^+ds)^\alpha]\leq C, \
\hat{\mathbb{E}}[|K_T^n|^\alpha]\leq C, \  \hat{\mathbb{E}}[(\int_0^T |Z_s^n|^2ds)^{\frac{\alpha}{2}}]\leq C.
\end{displaymath}
\end{lemma}

\begin{proof}
The first estimate can be derived easily from Lemma \ref{th3}. Applying $G$-It\^{o}'s formula to $|Y_t^n|^2$, we have
\begin{displaymath}
|Y_0^n|^2+\int_0^T |Z_s^n|^2d\langle B\rangle_s=|\xi|^2+\int_0^T 2Y_s^nf(s,Y_s^n,Z_s^n)ds-\int_0^T 2Y_s^nZ_s^ndB_s-\int_0^T 2Y_s^nd(K_s^n-L_s^n).
\end{displaymath}
Consequently
\begin{displaymath}
(\int_0^T |Z_s^n|^2d\langle B\rangle_s)^{\frac{\alpha}{2}}\leq C\{|\xi|^\alpha+|\int_0^T Y_s^nf(s,Y_s^n,Z_s^n)ds|^\alpha+|\int_0^T Y_s^nZ_s^ndB_s|^\alpha+|\int_0^T 2Y_s^nd(K_s^n-L_s^n)|^\alpha\}.
\end{displaymath}
By Proposition \ref{the1.3} and simple calculation, we obtain
\begin{equation}\label{eq3}\begin{split}
\hat{\mathbb{E}}[(\int_0^T |Z_s^n|^2ds)^{\frac{\alpha}{2}}]\leq &C\{\hat{\mathbb{E}}[\sup_{t\in[0,T]}|Y_t^n|^\alpha]+(\hat{\mathbb{E}}[\sup_{t\in[0,T]}|Y_t^n|^\alpha])^{1/2}[(\hat{\mathbb{E}}[|K_T^n|^\alpha])^{1/2}\\
&+(\hat{\mathbb{E}}[|L_T^n|^\alpha])^{1/2}+(\hat{\mathbb{E}}[(\int_0^T |f(s,0,0)|ds)^\alpha])^{1/2}]\}.
\end{split}\end{equation}
On the other hand,
\begin{displaymath}
K_T^n=\xi-Y_0^n+\int_0^T f(s,Y_s^n,Z_s^n)ds-\int_0^T Z_s^ndB_s+L_T^n.
\end{displaymath}
An easy computation shows that
\begin{equation}\label{eq4}
\hat{\mathbb{E}}[|K_T^n|^\alpha]\leq C\{\hat{\mathbb{E}}[\sup_{t\in[0,T]}|Y_t^n|^\alpha]+\hat{\mathbb{E}}[|L_T^n|^\alpha]
+\hat{\mathbb{E}}[(\int_0^T |f(s,0,0)|ds)^\alpha]+\hat{\mathbb{E}}[(\int_0^T |Z_s^n|^2ds)^{\frac{\alpha}{2}}]\}.
\end{equation}
Combining inequalities \eqref{eq3} and \eqref{eq4}, we can easily see the desired results.
\end{proof}

\begin{remark}\label{th6}
Set $A_t^n=L_t^n-K_t^n$. Then $(A_t^n)_{t\in[0,T]}$ is a process with finite variation . Moreover, it is easy to check that $(-A_t^n)_{t\in[0,T]}$ is a $G$-submartingale.  We denote by $Var(A^n)$ the total variation for $(A^n)$ on $[0,T]$. Then there exists a constant $C$ independent of $n$, such that for $1<\alpha<\beta$
\begin{displaymath}
\hat{\mathbb{E}}[|Var(A^n)|^\alpha]\leq C\{\hat{\mathbb{E}}[|L_T^n|^\alpha]+\hat{\mathbb{E}}[|K_T^n|^\alpha]\}\leq C.
\end{displaymath}
\end{remark}

We now show that the sequences $(Y^n)_{n=1}^\infty$, $(Z^n)_{n=1}^\infty$ and $(A^n)_{n=1}^\infty$ are convergent.
 \begin{lemma}\label{th7}
For $m,n\in\mathbb{N}$, set $\hat{Y}_t= Y_t^n-Y^m_t$, $\hat{Z}_t=Z^n_t-Z^m_t$ and $\hat{A}_t=A_t^n-A_t^m$. Then for any $2\leq \alpha<\beta$, we have
\begin{align}
\lim_{m,n\rightarrow \infty}\hat{\mathbb{E}}[\sup_{t\in[0,T]}|\hat{Y}_t|^\alpha]=0,\ \
\lim_{m,n\rightarrow \infty}\hat{\mathbb{E}}[(\int_0^T |\hat{Z}_s|^2 ds)^{\frac{\alpha}{2}}]=0,\ \  \lim_{m,n\rightarrow \infty} \hat{\mathbb{E}}[\sup_{t\in[0,T]}|\hat{A}_t|^\alpha]=0.\label{eqnx11}
\end{align}
\end{lemma}

\begin{proof}
The convergence property for $(Y^n)_{n=1}^\infty$ can be proved in a similar way as the proof of Lemma 4.4 in \cite{lp}. For reader's convenience, we give a brief proof here.
Without loss of generality, we may assume $S\equiv 0$ in \eqref{eq1}. Set $\hat{L}_t=L_t^n-L_t^m$, $\hat{K}_t=K_t^n-K_t^m$, $\hat{f}_t=f(t,Y_t^n,Z_t^n)-f(t,Y_t^m,Z_t^m)$ and $\bar{Y}_t=|\hat{Y}_t|^2$. By applying It\^{o}'s formula to $\bar{Y}_t^{\alpha/2}e^{rt}$, where $r$ is a constant to be determined later, we  get
\begin{equation}\label{eq1.5}
\begin{split}
&\quad\bar{Y}_t^{\alpha/2}e^{rt}+\int_t^T re^{rs}\bar{Y}_s^{\alpha/2}ds+\int_t^T \frac{\alpha}{2} e^{rs}
\bar{Y}_s^{\alpha/2-1}(\hat{Z}_s)^2d\langle B\rangle_s\\
&=
\alpha(1-\frac{\alpha}{2})\int_t^Te^{rs}\bar{Y}_s^{\alpha/2-2}(\hat{Y}_s)^2(\hat{Z}_s)^2d\langle B\rangle_s
+\int_t^T\alpha e^{rs}\bar{Y}_s^{\alpha/2-1}\hat{Y}_sd\hat{L}_s\\
&\quad+\int_t^T{\alpha} e^{rs}\bar{Y}_s^{\alpha/2-1}\hat{Y}_s\hat{f}_sds-\int_t^T\alpha e^{rs}\bar{Y}_s^{\alpha/2-1}(\hat{Y}_s\hat{Z}_sdB_s+\hat{Y}_sd\hat{K}_s)\\
&\leq
\alpha(1-\frac{\alpha}{2})\int_t^Te^{rs}\bar{Y}_s^{\alpha/2-2}(\hat{Y}_s)^2(\hat{Z}_s)^2d\langle B\rangle_s+\int_t^T{\alpha} e^{rs}\bar{Y}_s^{\frac{\alpha-1}{2}}|\hat{f}_s|ds\\
&\quad-\int_t^T\alpha e^{rs}\bar{Y}_s^{\alpha/2-1}(Y_s^n)^+dL^m_s-\int_t^T\alpha e^{rs}\bar{Y}_s^{\alpha/2-1}(Y_s^m)^+dL^n_s-(M_T-M_t),
\end{split}
\end{equation}
where $M_t=\int_0^t \alpha e^{rs}\bar{Y}_s^{\alpha/2-1}(\hat{Y}_s\hat{Z}_sdB_s+(\hat{Y}_s)^+dK_s^m+(\hat{Y}_s)^-dK_s^n)$. By Lemma 3.4 in \cite{HJPS1}, $\{M_t\}$ is a $G$-martingale. Applying the assumption on $f$ and the H\"{o}lder inequality, we obtain
\begin{displaymath}
\int_t^T{\alpha} e^{rs}\bar{Y}_s^{\frac{\alpha-1}{2}}|\hat{f}_s|ds
\leq (\alpha L+\frac{\alpha L^2}{\underline{\sigma}^2(\alpha-1)})\int_t^T e^{rs}\bar{Y}_s^{\alpha/2}ds
+\frac{\alpha(\alpha-1)}{4}\int_t^Te^{rs}\bar{Y}_s^{\alpha/2-1}(\hat{Z}_s)^2d\langle B\rangle_s.
\end{displaymath}
Let $r=1+\alpha L+\frac{\alpha L^2}{\underline{\sigma}^2(\alpha-1)}$. The above analysis shows that
\[\bar{Y}_t^{\alpha/2}e^{rt}+(M_T-M_t)\leq-\int_t^T\alpha e^{rs}\bar{Y}_s^{\alpha/2-1}(Y_s^n)^+dL^m_s-\int_t^T\alpha e^{rs}\bar{Y}_s^{\alpha/2-1}(Y_s^m)^+dL^n_s.\]
Then taking conditional expectation on both sides of the above inequality, we conclude that
\begin{displaymath}\begin{split}
|\hat{Y}_t|^{\alpha}\leq&
C\hat{\mathbb{E}}_t[-\int_t^T\bar{Y}_s^{\alpha/2-1}(Y_s^n)^+dL^m_s-\int_t^T\bar{Y}_s^{\alpha/2-1}(Y_s^m)^+dL^n_s]\\
\leq &C(n+m)\hat{\mathbb{E}}_t[\int_0^T|(Y_s^n)^+|^{\alpha-1}(Y_s^m)^+ds+\int_0^T|(Y_s^m)^+|^{\alpha-1}(Y_s^n)^+ds].
\end{split}\end{displaymath}
By Lemma \ref{th2}-\ref{th4}, Theorem \ref{the1.2} and the H\"{o}lder inequality, we have for some $2\leq \alpha<\beta$,
\[\lim_{n,m\rightarrow \infty}\hat{\mathbb{E}}[\sup_{t\in[0,T]}|Y_t^n-Y_t^m|^\alpha]=0.\]
Choosing $\alpha=2$ and $r=0$ in \eqref{eq1.5}, we get
\[|\hat{Y}_0|^2+\int_0^T|\hat{Z}_s|^2 d\langle B\rangle_s
=\int_0^T2\hat{Y}_s\hat{f}_sds-\int_0^T 2\hat{Y}_sd\hat{K}_s
+\int_0^T2\hat{Y}_sd\hat{L}_s-\int_0^T 2\hat{Y}_s\hat{Z}_s dB_s.\]
Observe that
\[\int_0^T2\hat{Y}_s\hat{f}_sds\leq 2L\int_0^T (|\hat{Y}_s|^2+|\hat{Y}_s||\hat{Z}_s|)ds\leq (2L+L^2/\varepsilon)\int_0^T|\hat{Y}_s|^2ds+\varepsilon\int_0^T|\hat{Z}_s|^2ds,\]
where $\varepsilon<\underline{\sigma}^2$. The above two equations yield that
\[\int_0^T|\hat{Z}_s|^2 ds
\leq C(\int_0^T|\hat{Y}_s|^2ds-\int_0^T \hat{Y}_sd\hat{K}_s
+\int_0^T\hat{Y}_sd\hat{L}_s-\int_0^T \hat{Y}_s\hat{Z}_s dB_s).\]
Applying Lemma \ref{th2}, Lemma \ref{th4}, Proposition \ref{the1.3} and the H\"{o}lder inequality, we derive that
\begin{align*}
\hat{\mathbb{E}}[(\int_0^T|\hat{Z}_s|^2)^{\frac{\alpha}{2}}]
&\leq C \{\hat{\mathbb{E}}[\sup_{t\in[0,T]}|\hat{Y}_t|^\alpha+\sup_{t\in[0,T]}|\hat{Y}_t|^{\frac{\alpha}{2}}(\lambda^{n,m}_T)^{\frac{\alpha}{2}}]
+\hat{\mathbb{E}}[(\int_0^T \hat{Y}_s^2\hat{Z}_s^2 ds)^{\frac{\alpha}{4}}]\}\\
&\leq C\{\hat{\mathbb{E}}[\sup_{t\in[0,T]}|\hat{Y}_t|^\alpha]+(\hat{\mathbb{E}}[\sup_{t\in[0,T]}|\hat{Y}_t|^\alpha])^{\frac{1}{2}}\}
+\frac{C^2}{2}\hat{\mathbb{E}}[\sup_{t\in[0,T]}|\hat{Y}_t|^\alpha]+\frac{1}{2}\hat{\mathbb{E}}[(\int_0^T|\hat{Z}_s|^2)^{\frac{\alpha}{2}}],
\end{align*}
where $\lambda_T^{n,m}=|L_T^n|+|L_T^m|+|K_T^n|+|K_T^m|$. It follows that
\begin{displaymath}
\lim_{n,m\rightarrow \infty}\hat{\mathbb{E}}[(\int_0^T|Z_s^n-Z_s^m|^2ds)^{\frac{\alpha}{2}}]=0.
\end{displaymath}
From Proposition \ref{the1.3} and the assumption of $f$, we have
\begin{displaymath}
\begin{split}
\hat{\mathbb{E}}[\sup_{t\in[0,T]}|A_t^n-A_t^m|^\alpha]
&\leq C\hat{\mathbb{E}}[\sup_{t\in[0,T]}|\hat{Y}_t|^\alpha+(\int_0^T|\hat{f}_s| ds)^\alpha
+\sup_{t\in[0,T]}|\int_0^t \hat{Z}_sdB_s|^\alpha]\\
&\leq C\{\hat{\mathbb{E}}[\sup_{t\in[0,T]}|\hat{Y}_t|^\alpha]+\hat{\mathbb{E}}[(\int_0^T|\hat{Z}_s|^2ds)^{\alpha/2}]\}\rightarrow 0.
\end{split}
\end{displaymath}
Using the convergence property of $(Y^n)_{n=1}^\infty$ and $(Z^n)_{n=1}^\infty$, it is easy to check that
\begin{displaymath}
\lim_{n,m\rightarrow \infty}\hat{\mathbb{E}}[\sup_{t\in[0,T]}|A_t^n-A_t^m|^{\alpha}]=0.
\end{displaymath}
\end{proof}



We now turn to the proof of Theorem \ref{th1}.

\begin{proof}
According to Lemma \ref{th7}, for any $2\leq \alpha<\beta$, there exists a triple $(Y,Z,A)\in \mathbb{S}_G^\alpha(0,T)$, such that
\[\lim_{n\rightarrow \infty}\hat{\mathbb{E}}[\sup_{t\in[0,T]}|Y_t^n-Y_t|^\alpha]=0,\ \
\lim_{n\rightarrow \infty}\hat{\mathbb{E}}[(\int_0^T |Z_s^n-Z_s|^2 ds)^{\frac{\alpha}{2}}]=0,\ \  \lim_{n\rightarrow \infty} \hat{\mathbb{E}}[\sup_{t\in[0,T]}|A_t^n-A_t|^\alpha]=0.\]
We then show that $(Y,Z,A)$ is a solution of the reflected $G$-BSDE with an upper obstacle. The fact that $Y$ is below the obstacle process $S$ can be derived easily from Lemma \ref{th3}. It remains to check that $\{-\int_0^t (S_s-Y_s)dA_s\}_{t\in[0,T]}$ is a decreasing $G$-martingale. Set \begin{displaymath}
\widetilde{K}_t^n:=\int_0^t (S_s-Y_s)dK_s^n.
\end{displaymath}
Since $S-Y$ is a nonnegative process in $S_G^\alpha(0,T)$, by Lemma 3.4 in \cite{HJPS1}, $\widetilde{K}^n$ is a decreasing $G$-martingale. Note that
\begin{align*}
\sup_{t\in[0,T]}|-\int_0^t (S_s-Y_s)dA_s-\widetilde{K}_t^n|
\leq &\sup_{t\in[0,T]}\{|\int_0^t Y_sdA_s-\int_0^tY_sdA_s^n|+|\int_0^t (Y_s -Y_s^n)dA_s^n|\\
&+|\int_0^t (Y_s -Y_s^n) dK_s^n|+|\int_0^t -(S_s-Y_s^n) dL_s^n|\}\\
\leq &\sup_{t\in[0,T]}\{|\int_0^t\widetilde{Y}_s^m d(A_s^n -A_s)|+|\int_0^t (Y_s-\widetilde{Y}_s^m)d(A_s^n -A_s)|\}\\
&+\sup_{t\in[0,T]}|Y_s-Y_s^n|[|var(A^n)|+|K_T^n|]+\sup_{t\in[0,T]}(Y_s^n-S_s)^+|L_T^n|,
\end{align*}
where $\widetilde{Y}_t^m=\sum_{i=0}^{m-1}Y_{t_i^m} I_{[t_{i}^m,t_{i+1}^m)}(t)$ and $t_i^m=\frac{iT}{m}$, $i=0,1,\cdots,m$. Recalling Lemma \ref{th2}-\ref{th7}, by a similar analysis as in the proof of Theorem 5.1 in \cite{lp}, we have
\begin{displaymath}
\lim_{n\rightarrow\infty}\hat{\mathbb{E}}[\sup_{t\in[0,T]}|-\int_0^t (S_s-Y_s)dA_s-\widetilde{K}_t^n|]\leq C(\hat{\mathbb{E}}[\sup_{t\in[0,T]}|Y_s-\widetilde{Y}_s^m|^2])^{1/2}.
\end{displaymath}
Applying Lemma 3.2 in \cite{HJPS1} and letting $m\rightarrow\infty$, it follows that
\[\lim_{n\rightarrow\infty}\hat{\mathbb{E}}[\sup_{t\in[0,T]}|-\int_0^t (S_s-Y_s)dA_s-\widetilde{K}_t^n|]=0,\]
which implies that $\{-\int_0^t (S_s-Y_s)dA_s\}$ is a decreasing $G$-martingale.

In the following, we prove that the solution constructed by the penalization procedure is the largest one. Suppose that $(Y',Z',A')$ is the solution of the reflected $G$-BSDE with parameters $(\xi,f,S)$ and $Y'_t\leq S_t$, $0\leq t\leq T$, we have
\begin{displaymath}
Y'_t=\xi+\int_t^T f(s,Y'_s,Z'_s)ds-\int_t^T n(Y'_s-S_s)^+ds-\int_t^T Z'_sdB_s+(A'_T-A'_t).
\end{displaymath}
Comparing with $G$-BSDE \eqref{eq1} and applying Theorem \ref{th12}, we can easily check that for all $n\in\mathbb{N}$, $Y'_t\leq Y_t^n$. Letting $n\rightarrow \infty$, we conclude that $Y'_t\leq Y_t$.
\end{proof}

\begin{remark}\label{th8}
The assumption (A3) and (A4) on $S$ and $\xi$ can be weakened in the following sense:
\begin{description}
\item[(A5)] There exist $\{\xi^n\}_{n\in\mathbb{N}}\subset L_G^\beta(\Omega_T)$ and sequence $\{S^n\}_{n\in\mathbb{N}}$ of $G$-It\^{o} processes
\begin{displaymath}
S^n_t=S^n_0+\int_0^t b^n(s)ds+\int_0^t l^n(s)d\langle B\rangle_s+\int_0^t \sigma^n(s)dB_s,
\end{displaymath}
with $\{b^n(t)\}$, $\{l^n(t)\}$ belong to $M_G^\beta(0,T)$ and $\{\sigma^n(t)\}$ belong to $H_G^\beta(0,T)$ for all $n\in\mathbb{N}$. Furthermore,  $\sup_{n\in\mathbb{N}}\hat{\mathbb{E}}[\sup_{t\in[0,T]}\{|b^n(t)|^\beta+|l^n(t)|^\beta+|\sigma^n(t)|^\beta\}]<\infty$, $\xi^n\leq S^n_T$ and $\xi^n\rightarrow \xi$, $\sup_{t\in[0,T]}|S_t^n-S_t|\rightarrow 0$ both quasi-surely and in $L_G^\beta(\Omega_T)$ as $n\rightarrow\infty$.
\end{description}

Under (A1), (A2) and (A5),  we need to consider the following family of $G$-BSDEs parameterized by $n=1,2,\ldots$.
\begin{displaymath}
Y_t^n=\xi^n+\int_t^T f(s,Y_s^n,Z_s^n)ds-n\int_t^T(Y_s^n-S_s^n)^+ds-\int_t^T Z_s^ndB_s-(K_T^n-K_t^n).
\end{displaymath}

Similar analysis as above, the reflected $G$-BSDE with parameters $(\xi,f,S)$ has at least one solution.
\end{remark}

\begin{remark}\label{th13}
If we further assume that the process $A$ satisfies the following condition:
\begin{description}
\item[(iv)] $A_t=A_t^1-A_t^2$, $t\in[0,T]$, where $A^i\in S_G^\alpha(0,T)$, $i=1,2$, $-A^1$ is a decreasing $G$-martingale, $A^2$ is an increasing process such that $\int_0^T (S_s-Y_s)dA_s^2=0$.
\end{description}
Then the solution satisfying (i), (ii) and (iv) of the reflected $G$-BSDE with parameters $(\xi,f,S)$ is unique.

Assume that $(Y,Z,A)$ and $(\tilde{Y},\tilde{Z},\tilde{A})$ are solutions of the reflected $G$-BSDE satisfying (i), (ii) and (iv). Let $\hat{Y}_t=Y_t-\tilde{Y}_t$, $\hat{Z}_t=Z_t-\tilde{Z}_t$, $\hat{f}_t=f(t,Y_t,Z_t)-f(t,\tilde{Y}_t,\tilde{Z}_t)$, $\hat{A}_t=A_t-\tilde{A}_t$.  For any $r,\varepsilon>0$, applying It\^{o}'s formula to $\bar{Y}_t^{\frac{\alpha}{2}}e^{rt}=(|\hat{Y}_t|^2+\varepsilon_\alpha)^{\frac{\alpha}{2}}e^{rt}$, where $\varepsilon_\alpha=\varepsilon(1-\alpha/2)^+$, we get
\begin{displaymath}
\begin{split}
&\quad \bar{Y}_t^{\alpha/2}e^{rt}+\int_t^T re^{rs}\bar{Y}_s^{\alpha/2}ds+\int_t^T \frac{\alpha}{2} e^{rs}
\bar{Y}_s^{\alpha/2-1}(\hat{Z}_s)^2d\langle B\rangle_s\\
&=
\alpha(1-\frac{\alpha}{2})\int_t^Te^{rs}\bar{Y}_s^{\alpha/2-2}(\hat{Y}_s)^2(\hat{Z}_s)^2d\langle B\rangle_s
+\int_t^T{\alpha} e^{rs}\bar{Y}_s^{\alpha/2-1}\hat{Y}_s\hat{f}_sds\\
&\quad +\int_t^T\alpha e^{rs}\bar{Y}_s^{\alpha/2-1}\hat{Y}_sd\hat{A}_s-\int_t^T\alpha e^{rs}\bar{Y}_s^{\alpha/2-1}\hat{Y}_s\hat{Z}_sdB_s.
\end{split}
\end{displaymath}
From the assumption of $f$, we have
\begin{align*}
&\int_t^T{\alpha} e^{rs}\bar{Y}_s^{\alpha/2-1}\hat{Y}_s\hat{f}_sds
\leq\int_t^T{\alpha}e^{rs}\bar{Y}_s^{\frac{\alpha-1}{2}}L(|\hat{Y}_s|+|\hat{Z}_s|)ds\\
\leq &(\alpha L+\frac{\alpha^2L}{\underline{\sigma}^2(\alpha-1)})\int_t^T e^{rs}\bar{Y}_s^{\frac{\alpha}{2}}ds+\frac{\alpha(\alpha-1)}{4}\int_t^T e^{rs}
\bar{Y}_s^{\alpha/2-1}(\hat{Z}_s)^2d\langle B\rangle_s.
\end{align*}
By condition (iv), it is easy to check that
\begin{align*}
&\int_t^T\alpha e^{rs}\bar{Y}_s^{\alpha/2-1}\hat{Y}_sd\hat{A}_s\\
=&\int_t^T\alpha e^{rs}\bar{Y}_s^{\alpha/2-1}\hat{Y}_sd(A_s^1-\tilde{A}^1_s)+\int_t^T\alpha e^{rs}\bar{Y}_s^{\alpha/2-1}\hat{Y}_sd(\tilde{A}^2_s-A_s^2)\\
\leq &\int_t^T\alpha e^{rs}\bar{Y}_s^{\alpha/2-1}(\hat{Y}_s)^-d\tilde{A}^1_s+\int_t^T\alpha e^{rs}\bar{Y}_s^{\alpha/2-1}(\hat{Y}_s)^+dA^1_s.
\end{align*}
Let $M_t=\int_0^t\alpha e^{rs}\bar{Y}_s^{\alpha/2-1}\hat{Y}_s\hat{Z}_sdB_s-\int_0^t\alpha e^{rs}\bar{Y}_s^{\alpha/2-1}(\hat{Y}_s)^-d\tilde{A}^1_s-\int_0^t\alpha e^{rs}\bar{Y}_s^{\alpha/2-1}(\hat{Y}_s)^+dA^1_s$. Then it is a $G$-martingale. Let $r=\alpha L+\frac{\alpha^2L}{\underline{\sigma}^2(\alpha-1)}+1$, we have
\begin{displaymath}
\bar{Y}_t^{\alpha/2}e^{rt}+(M_T-M_t)\leq 0.
\end{displaymath}
Taking conditional expectations on both sides, it follows that $Y\equiv\tilde{Y}$.  By applying It\^{o}'s formula to $(Y_t-\tilde{Y}_t)^2(\equiv0)$ on $[0,T]$ and taking expectations, we get
\begin{displaymath}
\hat{\mathbb{E}}[\int_0^T (Z_s-\tilde{Z}_s)^2d\langle B\rangle_s]=0,
\end{displaymath}
which implies $Z\equiv \tilde{Z}$. Then it is easy to check that $A\equiv \tilde{A}$.

\end{remark}





\end{document}